\documentclass{amsart}          

\usepackage{amsmath,amsthm}     
\usepackage{amssymb}            
\usepackage{euscript}           
\usepackage{enumerate,calc}
\usepackage[matrix,arrow,curve,frame]{xy}    

\xymatrixcolsep{1.9pc}                          
\xymatrixrowsep{1.9pc}
\newdir{ >}{{}*!/-5pt/\dir{>}}                  

\numberwithin{equation}{section}

\newtheorem{thm}[subsection]{Theorem}

\newtheorem{prop}[subsection]{Proposition}

\newtheorem{cor}[subsection]{Corollary}

\theoremstyle{definition}  

\newtheorem*{lemma*}{Lemma}
\newtheorem*{thm*}{Theorem}

\newtheorem{lem}[subsection]{Lemma}{\bf}{\it}

  {\end{list}}

\newcommand{\chapter}{\section}

\newcommand{\Smash}             {\wedge}

\newcommand{\cat}{\EuScript}    

\newcommand{\cU}{{\cat U}}

\newcommand{\field}[1]  {\mathbb #1} 

\newcommand{\R}         {\field R}
\newcommand{\N}         {\field N}

\newcommand{\Z}         {\field Z}

\newcommand{\bP}        {\field P}

\newcommand{\ra}{\rightarrow}                   

\newcommand{\fib}{\twoheadrightarrow}           

\newcommand{\inc}{\hookrightarrow}              

\newcommand{\tuborg}{\left\{\begin{array}{ll}}
\newcommand{\sluttuborg}{\end{array}\right.}

\begin{document}

\title{A Freeness Theorem for $RO(\Z/2)$-graded Cohomology}

\author{William C. Kronholm}
\address{Department of Mathematics and Statistics\\ Swarthmore College\\ Swarthmore, PA 19081 }

\date{\today}
\begin{abstract}
In this paper it is shown that the $RO(\Z/2)$-graded cohomology of a
certain class of $\text{Rep}(\Z/2)$-complexes, which includes
projective spaces and Grassmann manifolds, is always free as a
module over the cohomology of a point when the coefficient Mackey
functor is $\underline{\Z/2}$.

\end{abstract}
\maketitle

\tableofcontents


\section{Introduction}
\label{intro} In nonequivariant topology, it is a triviality that
spaces built of only even dimensional cells will have free
cohomology, regardless of the chosen coefficient ring.  It is just
as easy to see that every space has free cohomology when the
coefficient ring is taken to be $\Z/2$.  Analogous results are not
so clear in the equivariant setting.

In \cite{FL}, it is shown that the $RO(\Z/p)$-graded homology of a
$\Z/p$-space built of only even dimensional cells is free as a
module over the homology of a point, regardless of which Mackey
functor is chosen for coefficients.  The goal of this paper is to establish a similar result for the cohomology of $G=\Z/2$-spaces without the restriction to cells of even degrees, but with the assumption of using
constant $\underline{\Z/2}$ Mackey functor coefficients.  Here is
the main result:

\begin{thm*}
If $X$ is a connected, locally finite, finite
dimensional $\text{Rep}(\Z/2)$-complex, then
$H^{*,*}(X;\underline{\Z/2})$ is free as a
$H^{*,*}(pt;\underline{\Z/2})$-module.
\end{thm*}

\noindent (The bigrading will be explained in Section \ref{sec:Prelim}.)

The projective spaces and Grassmann
manifolds associated to representations of $\Z/2$ are examples of such $\text{Rep}(\Z/2)$-complexes.  In these
particular cases, the free generators of the cohomology modules are
in bijective correspondence with the Schubert cells.  The precise degrees of the cohomology generators is typically unknown, much
like in \cite{FL}.

Section \ref{sec:Prelim} provides some of the background and notation required
for the rest of the paper.  Most of this information can be found in
\cite{Alaska} and \cite{FL} but is reproduced here for convenience.
Section \ref{sec:Freeness} holds the main freeness theorem. As
applications of the freeness theorem, section \ref{sec:RPs} exhibits
some techniques for calculating the cohomology of
$\text{Rep}(G)$-complexes.  The importance of such calculations lies
in their potential applications toward understanding $RO(G)$-graded
equivariant characteristic classes.

The work in this paper was originally part of the author's
dissertation while at the University of Oregon.

The author is indebted to Dan Dugger for his guidance and
innumerable helpful conversations.

\section{Preliminaries}
\label{sec:Prelim} This section contains some of the basic machinery
and notations that will be used throughout the paper.  In this
section, $G$ can be any finite group unless otherwise specified.

Given a $G$-representation $V$, let $D(V)$ and $S(V)$ denote the
unit disk and unit sphere, respectively, in $V$ with action induced
by that on $V$.  A \bf$\text{Rep}(G)$-complex \rm is a $G$-space $X$ with a
filtration $X^{(n)}$ where $X^{(0)}$ is a disjoint union of
$G$-orbits and $X^{(n)}$ is obtained from $X^{(n-1)}$ by attaching
cells of the form $D(V_\alpha)$ along maps $f_\alpha \colon
S(V_\alpha) \ra X^{(n-1)}$ where $V_\alpha$ is an $n$-dimensional
real representation of $G$.  The space $X^{(n)}$ is referred to as
the \bf $n$-skeleton \rm of $X$, and the filtration is referred to as a \bf cell
structure\rm .

For the precise definition of a Mackey functor when $G=\Z/2$, the
reader is referred to \cite{LMM} or \cite{DuggerKR}.  A summary
of the important aspects of a Mackey functor is given here.  The
data of a Mackey functor are encoded in a diagram like the one
below.

\[\xymatrix{ M(\Z/2) \ar@(ur,ul)[]^{t^*} \ar@/^/[r]^(0.6){i_*} & M(e) \ar@/^/[l]^(0.4){i^*}} \]

The maps must satisfy the following four conditions.
\begin{enumerate}
\item $(t^*)^2 = id$
\item $t^*i^*=i^*$
\item $i_*(t^*)^{-1}=i_*$
\item $i^*i_*=id+t^*$
\end{enumerate}

According to \cite{Alaska}, each Mackey functor $M$ uniquely
determines an $RO(G)$-graded cohomology theory characterized by
\begin{enumerate}
\item $H^n(G/H;M) =\begin{cases}
M(G/H) & \text{ if } n=0 \\
0 & \text{otherwise}\end{cases}$
\item The map $H^0(G/K;M) \ra H^0(G/H;M)$ induced by $i \colon G/H \ra G/K$ is the transfer map $i^*$ in the Mackey functor.
\end{enumerate}

A $p$-dimensional real $\Z/2$-representation $V$ decomposes as
$V=(\R^{1,0})^{p-q} \oplus (\R^{1,1})^q =\R^{p,q}$ where $\R^{1,0}$
is the trivial 1-dimensional real representation of $\Z/2$ and $\R^{1,1}$ is the nontrivial
1-dimensional real representation of $\Z/2$.  Thus the $RO(\Z/2)$-graded theory is
a bigraded theory, one grading measuring dimension and the other
measuring the number of ``twists''.  In this case, we write
$H^{V}(X;M)=H^{p,q}(X;M)$ for the $V^{\text{th}}$ graded component
of the $RO(\Z/2)$-graded equivariant cohomology of $X$ with
coefficients in a Mackey functor $M$.

In this paper, $G$ will typically be $\Z/2$ and the Mackey functor
will almost always be constant $M=\underline{\Z/2}$ which has the
following diagram.

\[\xymatrix{ \Z/2 \ar@(ur,ul)[]^{id} \ar@/^/[r]^{0} & \Z/2 \ar@/^/[l]^{id}} \]

With these constant coefficients, the $RO(\Z/2)$-graded cohomology
of a point is given by the picture in Figure \ref{fig:pt}.

\begin{figure}[htpb]
\centering
\begin{picture}(100,100)(-100,-100)
\put(-100,-50){\vector(1,0){100}}
\put(-50,-100){\vector(0,1){100}}

\put(-50, -51){\line(0,1){40}} \put(-50, -51){\line(1,1){40}}
\put(-50, -91){\line(0,-1){20}} \put(-50, -91){\line(-1,-1){20}}
\put(-50, -51){\circle*{2}}

\multiput(-90,-51)(20,0){5}{\line(0,1){3}}
\multiput(-52,-90)(0,20){5}{\line(1,0){3}}
\put(-49,-57){$\scriptscriptstyle{0}$}
\put(-29,-57){$\scriptscriptstyle{1}$}
\put(-9,-57){$\scriptscriptstyle{2}$}
\put(-69,-57){$\scriptscriptstyle{-1}$}
\put(-89,-57){$\scriptscriptstyle{-2}$}
\put(-57,-47){$\scriptscriptstyle{0}$}
\put(-57,-27){$\scriptscriptstyle{1}$}
\put(-57,-7){$\scriptscriptstyle{2}$}

\put(-48,-35){$\tau$} \put(-30,-35){$\rho$} \put(-10,-15){$\rho^2$}
\put(-30,-17){$\tau\rho$} \put(-58, -90){$\theta$} \put(-48,
-102){$\frac{\theta}{\tau}$} \put(-78, -102){$\frac{\theta}{\rho}$}

\put(-58,0){${q}$} \put(0,-60){${p}$}

\multiput(-50, -30)(20,20){2}{\circle*{2}} \multiput(-50,
-50)(20,20){3}{\circle*{2}} \put(-50, -10){\circle*{2}} \put(-50,
-90){\circle*{2}} \put(-50, -110){\circle*{2}} \put(-70,
-110){\circle*{2}}

\end{picture}

\caption{$H^{*,*}(pt;\Z/2)$} \label{fig:pt}
\end{figure}

Every lattice point in the picture that is inside the indicated
cones represents a copy of the group $\Z/2$.  The \bf top cone \rm is a
polynomial algebra on the nonzero elements $\rho \in
H^{1,1}(pt;\underline{\Z/2})$ and $\tau \in
H^{0,1}(pt;\underline{\Z/2})$.  The nonzero element $\theta \in H^{0,-2}(pt;\underline{\Z/2})$ in the \bf bottom
cone \rm is infinitely divisible by both $\rho$ and $\tau$.  The
cohomology of $\Z/2$ is easier to describe:
$H^{*,*}(\Z/2;\underline{\Z/2})=\Z/2[t, t^{-1}]$ where $t \in
H^{0,1}(\Z/2;\underline{\Z/2})$.  Details can be found in
\cite{DuggerKR} and \cite{Caruso}.

A useful tool is the following exact sequence of \cite{AM}.

\begin{lem}[Forgetful Long Exact Sequence]
\label{lemma:forget} Let $X$ be a based $\Z/2$-space.  Then for
every $q$ there is a long exact sequence

$$\xymatrix{\cdots \ar[r] & H^{p,q}(X) \ar[r]^(0.4){\cdot \rho} & H^{p+1,q+1}(X) \ar[r]^(0.55){\psi} & H^{p+1}_{sing}(X) \ar[r]^(0.45)\delta & H^{p+1,q}(X)} \ra \cdots$$

\end{lem}

\noindent The map $\cdot \rho$ is multiplication by $\rho \in
H^{1,1}(pt;\underline{\Z/2})$ and $\psi$ is the forgetful map to
non-equivariant cohomology with $\Z/2$ coefficients.

\section{The Freeness Theorem}
\label{sec:Freeness}

Computing the $RO(G)$-graded cohomology of a $G$-space $X$ is
typically quite a difficult task.  However, if $X$ has a filtration
$X^{(0)} \subseteq X^{(1)} \subseteq \cdots$, then we can take
advantage of the long exact sequences in cohomology arising from the cofiber
sequences $X^{(n)} \subseteq X^{(n+1)} \ra X^{(n+1)}/X^{(n)}$.
These long exact sequences paste together as an exact couple in the usual way,
giving rise to a spectral sequence associated to the filtration.

If $X$ is a $G$-CW complex or a $\text{Rep}(G)$-complex, then $X$
has a natural filtration coming from the cell structure.  In either
case, if $X$ is connected, the quotient spaces $X^{(n+1)}/X^{(n)}$
are wedges of $(n+1)$-spheres with action determined by the type of
cells that were attached.  Examples of this sort appear throughout the paper.

For the remainder of the paper, we
will only be interested in the case $G=\Z/2$ and always take coefficients to be $\underline{\Z/2}$.  These choices will be implicit in our notation.

Given a filtered $\Z/2$ space $X$, for each fixed $q$ there is a long exact sequence

\[ \cdots H^{*,q}(X^{(n+1)}/X^{(n)})\ra H^{*,q}(X^{(n+1)}) \ra H^{*,q}(X^{(n)}) \ra H^{*+1,q}(X^{(n+1)}/X^{(n)}) \cdots \]

\noindent and so there is one spectral sequence for each integer $q$.  The
specifics are given in the following proposition.

\begin{prop} Let $X$ be a filtered $\Z/2$-space.  Then for each $q\in\Z$ there is a spectral sequence with
$$E_1^{p,n} = H^{p,q}(X^{(n+1)}, X^{(n)})$$
converging to $H^{p,q}(X)$.
\end{prop}

The construction of the spectral sequence is completely standard. See, for example, Proposition 5.3 of \cite{McC}.

It is convenient to plot the $RO(\Z/2)$-graded cohomology in the
plane with $p$ along the horizontal axis and $q$ along the vertical
axis, and this turns out to be a nice way to view the cellular spectral
sequences as well. When doing so, the
differentials on each page of the spectral sequence have bidegree
$(1,0)$ in the plane, but reach farther up the filtration on each
page.  It is important to keep track of at
what stage of the filtration each group arises.  In practice, this can be done by using different colors for group that arise at different stages of the filtration.

It is often quite difficult to determine the effect of all of the
attaching maps in the cell attaching long exact sequences.  If $X$
is locally finite, then the cells can be attached one at a time, in
order of dimension. This simplicity will make it easier to analyze
the differentials in the spectral sequence of the `one at a time'
cellular filtration, even when the precise impact of the attaching maps are not a priori known.

\begin{lem}
 Let $B$ be a $\text{Rep}(\Z/2)$-complex with free cohomology that is built only of cells of dimension strictly less than $p$.  Suppose $X$ is obtained from $B$ by attaching a single $(p,q)$-cell and let $\nu$ denote the generator for the cohomology of $X/B \cong S^{p,q}$. Then after an appropriate change of basis either
\begin{enumerate}
\item all attaching maps to the top cone of $\nu$ are zero (that is, $d(a)=0$ for all $a$ with $a\in H^{*,q_a}(B)$ with $q_a \geq q-1$),
\item the cell attaching `kills' $\nu$ and a free generator in dimension $(p-1,q)$, or
\item all nonzero differentials hit the bottom cone of $\nu$.
\end{enumerate}
\end{lem}

\begin{proof}
Consider the cellular spectral sequence associated to attaching a single $(p,q)$ cell to $B$.  The effects of attaching such a cell can cause the lower dimensional
generators to hit either the `top cone' or the `bottom cone'
of the newly attached free generator $\nu$ of degree $(p,q)$.

Suppose first that all nonzero differentials hit the top
cone.  Then any free generator $\omega_i$ having a nonzero
differential in the spectral sequence must have degree $(p_i,q_i)$
where $p_i=p-1$ and $q_i \geq q$.  For illustrative purposes, the $E_1$ page of the cellular spectral sequence of an example of this type is pictured in
Figure \ref{fig:attach1pqcell}.  In this example, there are two generators $\omega_1$ and $\omega_2$ with bidegree $(p-1,q_1)$ and one generator $\omega_3$ with bidegree $(p-1,q_i)$.

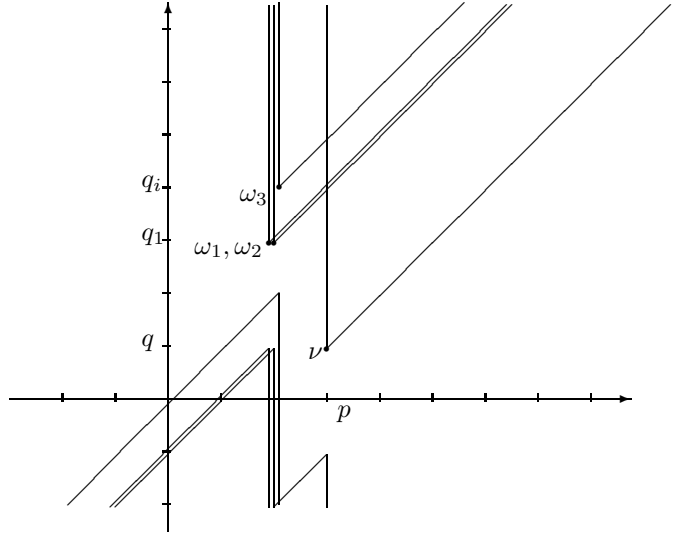
\begin{figure}[htpb]

\centering

\begin{picture}(330,230)(-150,-110)
\put(-110,-50){\vector(1,0){235}}
\put(-50,-100){\vector(0,1){200}}

\put(-8, 30){\line(0,1){70}} \put(-8, 30){\line(1,1){70}} \put(-8,
-10){\line(0,-1){80}} \put(-8, -10){\line(-1,-1){80}} \put(-8,
30){\circle*{2}}
\put(-23,25){$\omega_3$}

\put(-10, 9){\line(0,1){90}} \put(-10, 9){\line(1,1){90}} \put(-10,
-31){\line(0,-1){60}} \put(-10, -31){\line(-1,-1){60}} \put(-10,
9){\circle*{2}}

\put(-40,5){$\omega_1,\omega_2$}

\put(-12, 9){\line(0,1){90}} \put(-12, 9){\line(1,1){90}} \put(-12,
-31){\line(0,-1){60}} \put(-12, -31){\line(-1,-1){60}} \put(-12,
9){\circle*{2}}

\put(10, -31){\line(0,1){130}} \put(10, -31){\line(1,1){130}}
\put(10, -71){\line(0,-1){20}} \put(10, -71){\line(-1,-1){20}}
\put(10, -31){\circle*{2}}

\put(3, -35){$\nu$}

\multiput(-90,-51)(20,0){11}{\line(0,1){3}}
\multiput(-52,-90)(0,20){10}{\line(1,0){3}}

\put(14,-57){$p$} \put(-60,30){$q_i$} \put(-60,10){$q_1$}
\put(-60,-30){$q$}

\end{picture}

\caption{The $E_1$ page of the cellular spectral sequence attaching
a single $(p,q)$-cell to $B$.} \label{fig:attach1pqcell}

\end{figure}

Here, only the generator associated to the $(p,q)$-cell and the
generators with nonzero differentials are shown.  Each of the
$\omega_i$ satisfies $d(\omega_i)=\tau^{n_i} \nu$ for integers
$n_i$.  Relabeling if necessary, we can arrange so that the
$\omega_i$ satisfy $n_1 \leq n_2 \leq \cdots$.

Let $A=\langle\omega_i\rangle$, the $H^{*,*}(pt)$-span of the $\omega_i$'s. A change of basis can be performed
on $A$, after which we may assume $d(\omega_1)=\tau^{n_1}\nu$ and
$d(\omega_i)=0$ for $i>1$. Indeed,
$\{\tau^{n_i-n_1}\omega_1+\omega_i\}$ is a basis for $A$ and
$d(\tau^{n_i-n_1}\omega_1+\omega_i)=\tau^{n_1}\nu$ if $i=1$ and is
zero otherwise. (In effect, the attaching map can `slide' off of all the $\omega_i$
except for the one for which $q_i$ is minimal.)

If $\omega_1$ happens to be in dimension $(p-1,q)$, then the newly
attached cell `kills' $\omega_1$ and $\nu$.
Otherwise the nonzero portion of the spectral sequence is illustrated in Figure \ref{fig:changebasis}.

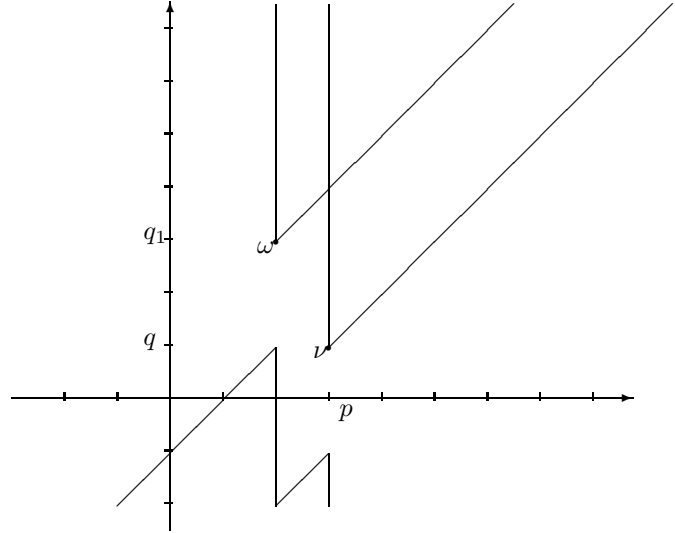
\begin{figure}[htpb]

\centering

\begin{picture}(330,230)(-150,-110)
\put(-110,-50){\vector(1,0){235}}
\put(-50,-100){\vector(0,1){200}}

\put(-10, 9){\line(0,1){90}} \put(-10, 9){\line(1,1){90}}
\put(-10,-31){\line(0,-1){60}} \put(-10, -31){\line(-1,-1){60}}
\put(-10, 9){\circle*{2}} \put(-17, 5){$\omega$}

\put(10, -31){\line(0,1){130}} \put(10, -31){\line(1,1){130}}
\put(10, -71){\line(0,-1){20}} \put(10, -71){\line(-1,-1){20}}
\put(10, -31){\circle*{2}}

\put(4, -35){$\nu$}

\multiput(-90,-51)(20,0){11}{\line(0,1){3}}
\multiput(-52,-90)(0,20){10}{\line(1,0){3}}

\put(14,-57){$p$} \put(-60,10){$q_1$} \put(-60,-30){$q$}

\end{picture}
\caption{The nonzero portion of the same spectral sequence, after a
change of basis.} \label{fig:changebasis}
\end{figure}

After taking cohomology, the spectral sequence collapses, as in Figure \ref{fig:E2page}.

\begin{figure}[htpb]

\centering

\begin{picture}(330,230)(-150,-110)
\put(-110,-50){\vector(1,0){235}}
\put(-50,-100){\vector(0,1){200}}

\put(-10, -30){\line(0,-1){60}} \put(-10, -30){\line(-1,-1){60}}
\put(-10, -70){\circle*{2}} \put(-10, -50){\circle*{2}}

\put(-8,-45){$\omega_1 \frac{\theta}{\tau^m}$}

\put(10, -30){\line(0,1){20}} \put(10,-10){\line(1,1){110}} \put(10,
-30){\line(1,1){130}} \put(10, -70){\line(0,-1){20}} \put(10,
-70){\line(-1,-1){20}} \put(10, -30){\circle*{2}}

\put(14, -35){$\nu$}

\multiput(-90,-51)(20,0){11}{\line(0,1){3}}
\multiput(-52,-90)(0,20){10}{\line(1,0){3}}

\put(14,-57){$p$} \put(-60,10){$q_1$} \put(-60,-30){$q$}

\end{picture}
\caption{The $E_2=E_\infty$ page of the above spectral sequence.}
\label{fig:E2page}
\end{figure}
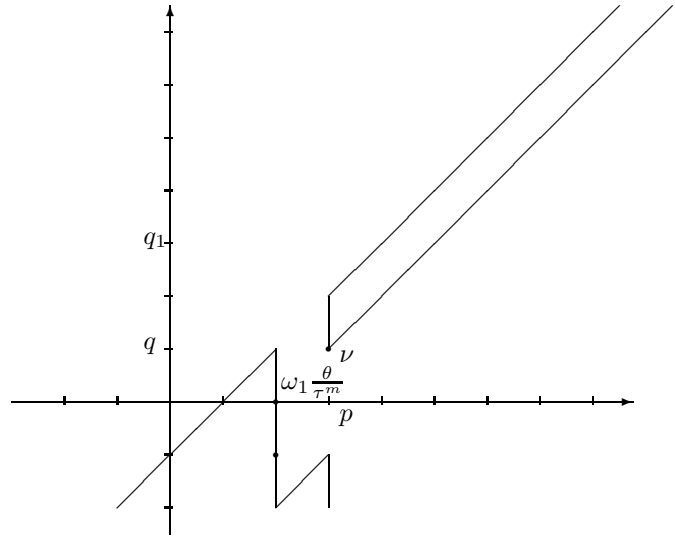

There is a class $\omega_1 \frac{\theta}{\tau^m}$ that, potentially,
could satisfy $\rho \cdot \omega_1 \frac{\theta}{\tau^m}=\nu$.
However, for degree reasons, $\rho \cdot \omega_1
\frac{\theta}{\tau^{m+1}}=0$ and since $\rho$ and $\tau$ commute, $\rho \cdot \omega_1 \frac{\theta}{\tau^m}=0$.  This
means $\nu$ determines a nonzero class in $H^{*,*}(X)$ that is not
in the image of $\cdot \rho$.  If $B$ is based, then $X$ is based,
and so, by the forgetful long exact sequence, $\nu$ determines a
nonzero class in non-equivariant cohomology. Then since $\tau$ maps to $1$
in non-equivariant cohomology, $\tau^n \nu$ is nonzero for all $n$.  But,
as the picture indicates, $\tau^n \nu$ is zero for large enough
$n$.  This contradiction implies that there could not have been any
nonzero differentials hitting the top cone of $\nu$.

This argument is independent of whether there are any differentials
hitting the bottom cone, and so there simply cannot be any nonzero
differentials on the top cone.\qed
\end{proof}

Differentials hitting the bottom cone can cause a shifting in degree of the cohomology generators.

\begin{thm} Suppose $X$ is a $\text{Rep}(\Z/2)$-complex formed by attaching a single $(p,q)$-cell to a space $B$.  Suppose also that $\tilde{H}^{*,*}(B)$ is a free $H^{*,*}(pt)$-module with a single generator $\omega$ of dimension strictly less than $p$.  Then $H^{*,*}(X)$ is a free $H^{*,*}(pt)$-module.  In particular, one of the following must hold:
\begin{enumerate}
\item $H^{*,*}(X) \cong H^{*,*}(pt)$.

\item $H^{*,*}(X) \cong H^{*,*}(B)\oplus \Sigma^\nu H^{*,*}(pt)$, where the degree of $\nu$ is $(p,q)$.

\item $H^{*,*}(X)$ is free with two generators $a$ and $b$.
\end{enumerate}

In (3) above, the degrees of the generators $a$ and $b$ are
$(p-n-1,q-n-1)$ and $(p,q-m-1)$ where
$d(\omega)=\frac{\theta}{\rho^n\tau^m}\nu$.

\end{thm}
\begin{proof}  Under these hypotheses, there is a cofiber sequence of the form $B \stackrel{i}\inc X \stackrel{j}\fib S^{p,q}$.  Denote by $\nu$ the generator of $H^{*,*}(S^{p,q})$.

If $d(\omega)=\nu$ then $(1)$ holds and $H^{*,*}(X)$ is free. If
$d(\omega)=0$, then $(2)$ holds and again $H^{*,*}(X)$ is free. The
remaining case is $d(\omega)\neq0$.  By the previous lemma, $d(\omega)$ is in the bottom cone of $\nu$. That is $d(\omega)=\frac{\theta}{\rho^n\tau^m}\nu$ for some $n$ and
$m$.  Recall that $\nu$ has dimension $(p,q)$ and so $\omega$ has
dimension $(p-n-1, q-n-m-2)$.  The $E_1$ page of the cellular
spectral sequence is given in Figure \ref{fig:E1bottom}.

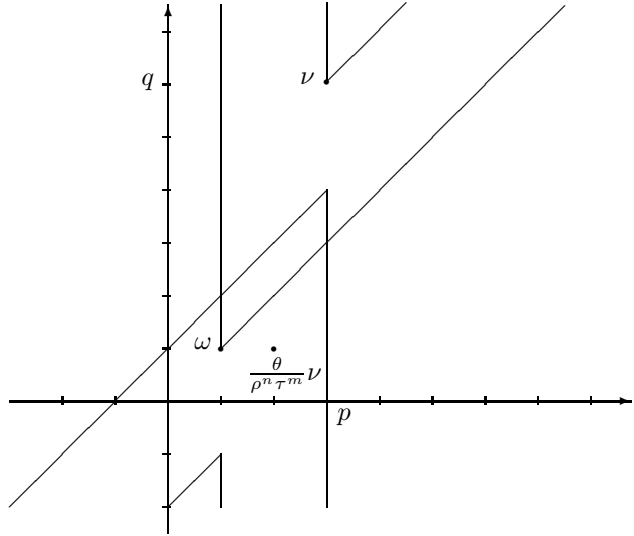
\begin{figure}[htpb]

\centering

\begin{picture}(330,230)(-150,-110)
\put(-110,-50){\vector(1,0){235}}
\put(-50,-100){\vector(0,1){200}}

\put(-30, -30){\line(0,1){130}} \put(-30, -30){\line(1,1){130}}
\put(-30, -70){\line(0,-1){20}} \put(-30, -70){\line(-1,-1){20}}
\put(-30, -30){\circle*{2}} \put(-40,-30){$\omega$}

\put(10, 71){\line(0,1){30}} \put(10, 71){\line(1,1){30}} \put(10,
30){\line(0,-1){120}} \put(10, 30){\line(-1,-1){120}} \put(10,
71){\circle*{2}} \put(0,70){$\nu$} \put(-10, -30){\circle*{2}}
\put(-20,-42){$\frac{\theta}{\rho^n\tau^m}\nu$}

\multiput(-90,-51)(20,0){11}{\line(0,1){3}}
\multiput(-52,-90)(0,20){10}{\line(1,0){3}}

\put(14,-57){$p$} \put(-60,70){$q$}

\end{picture}
\caption{The $E_1$ page of the cellular spectral sequence with a
single nonzero differential hitting the bottom cone of an attached
$(p,q)$-cell.} \label{fig:E1bottom}
\end{figure}

After taking cohomology, the spectral sequence collapses, and what
remains is pictured in Figure \ref{fig:E2bottom}.

\begin{figure}[htpb]

\centering

\begin{picture}(330,230)(-150,-110)
\put(-110,-50){\vector(1,0){235}}
\put(-50,-100){\vector(0,1){200}}

\put(-30, 30){\line(0,1){70}} \put(-30, 30){\line(1,1){70}}
\put(-30, 30){\circle*{2}} \put(-40,30){$a$} \put(10,
10){\line(1,1){90}} \put(10, 10){\circle*{2}} \put(0, 10){$b$}
\put(10, 30){\line(1,1){70}} \put(10, 30){\circle*{2}} \put(0,
30){$b_1$} \put(10, 50){\line(1,1){50}} \put(10, 50){\circle*{2}}
\put(0, 50){$b_2$} \put(-30, -70){\line(0,-1){20}} \put(-30,
-70){\line(-1,-1){20}}

\put(10, 71){\line(0,1){30}} \put(10, 71){\line(1,1){30}} \put(10,
71){\circle*{2}} \put(-3, 71){$b_m$} \put(10, -30){\line(0,-1){60}}
\put(10, -30){\circle*{2}} \put(-30, -10){\line(-1,-1){80}}
\put(-30, -10){\circle*{2}} \put(-30, -30){\line(-1,-1){60}}
\put(-30, -30){\circle*{2}} \put(-30, -50){\line(-1,-1){40}}
\put(-30, -50){\circle*{2}} \put(-10, -48){\line(-1,-1){40}}
\put(-10, -48){\line(0,-1){40}} \put(-10, -48){\circle*{2}}

\multiput(-90,-51)(20,0){11}{\line(0,1){3}}
\multiput(-52,-90)(0,20){10}{\line(1,0){3}}

\put(14,-57){$p$} \put(-60,70){$q$}

\end{picture}
\caption{The $E_2=E_\infty$ page of the cellular spectral sequence
with a single nonzero differential hitting the bottom cone of an
attached $(p,q)$-cell.} \label{fig:E2bottom}
\end{figure}

Let $a$ be the generator in degree $(p-n-1, q-n-1)$ and $b$ the generator in dimension
$(p,q-m-1)$. For degree reasons, $b$ is not in the image
of $\cdot \rho$ and so determines a nonzero class in non-equivariant
cohomology.  Thus, $\tau^i b$ is nonzero for all $i$, and so we have
that $b_i=\tau^i b$. In particular, $\rho^{n+1} a$ and $\tau^m b$
generate $H^{p,q}(X)$.  Consider the portion of the long exact
sequence associated to the cofiber sequence $B \stackrel{i}\inc X
\stackrel{j}\fib S^{p,q}$ given below:

$$\xymatrix{\cdots \ar[r] & H^{p,q}(S^{p,q}) \ar[r]^{j^*} & H^{p,q}(X) \ar[r]^{i^*} & H^{p,q}(B) \ar[r] & 0}$$

Since $i^*(\rho^{n+1} a)=i^*(\tau^m b)=\rho^{n+1}\tau^m\omega$,
exactness implies that $j^*(\nu)=\rho^{n+1} a + \tau^m b$.  Also
$j^*$ is an $H^{*,*}(pt)$-module homomorphism, and so
$j^*(\frac{\theta}{\rho^{n+1}}\nu)=\theta a$ and
$j^*(\frac{\theta}{\tau^m}\nu)=\theta b$.  In particular, we can
create a map $f$ from a free module with generators $\alpha$ and
$\beta$ in degrees $(p-n-1,q-n-1)$ and $(p,q-m-1)$ respectively
to $\tilde{H}^{p,q}(X)$ with $f(\alpha)=a$ and $f(\beta)=b$.  This
$f$ is an isomorphism.

\end{proof}

\begin{thm}[Freeness Theorem]
\label{thm:freeness} If $X$ is a connected, locally finite, finite
dimensional $\text{Rep}(\Z/2)$-complex, then
$H^{*,*}(X;\underline{\Z/2})$ is free as a
$H^{*,*}(pt;\underline{\Z/2})$-module.
\end{thm}

\begin{proof}

Since $X$ is locally finite, the cells can be attached one at a
time.  Order the cells $\alpha_1, \alpha_2, \dots$ so that their
degrees satisfy $p_i \leq p_j$ if $i\leq j$ and $q_i \leq q_j$ if
$p_i=p_j$ and $i \leq j$. We can proceed by induction over the
spaces in the `one-at-a-time' cell filtration $X^{(0)} \subseteq \cdots \subseteq X^{(n)}
\subseteq \cdots \subseteq X$, with the base case obvious since $X$
is connected.

First, suppose that $H^{*,*}(X^{(n)})$ is a free
$H^{*,*}(pt)$-module and that $X^{(n+1)}$ is obtained by attaching a
single $(p,q)$-cell and that $X^{(n)}$ has no $p$-cells.  Denote by
$\nu$ the free generator of $H^{*,*}(X^{(n+1)}/X^{(n)}) \cong
H^{*,*}(S^{p,q})$.  Consider the spectral sequence of the filtration
$X^{(n)} \subseteq X^{(n+1)}$.  An example is pictured below in Figure
\ref{fig:free1} to aid in the discussion.

\begin{figure}[htpb]

\centering

\begin{picture}(330,230)(-150,-110)

\put(-10, -10){\line(0,1){110}} \put(-10, -10){\line(1,1){110}}
\put(-10, -50){\line(0,-1){40}} \put(-10, -50){\line(-1,-1){40}}
\put(-10, -10){\circle*{2}}

\put(-8, 50){\line(0,1){50}} \put(-8, 50){\line(1,1){50}} \put(-8,
10){\line(0,-1){100}} \put(-8, 10){\line(-1,-1){100}} \put(-8,
50){\circle*{2}}

\put(-10, 70){\line(0,1){30}} \put(-10, 70){\line(1,1){30}}
\put(-10, 30){\line(0,-1){120}} \put(-10, 30){\line(-1,-1){120}}
\put(-10, 70){\circle*{2}}

\put(-12, -30){\line(0,1){130}} \put(-12, -30){\line(1,1){130}}
\put(-12, -70){\line(0,-1){20}} \put(-12, -70){\line(-1,-1){20}}
\put(-12, -30){\circle*{2}}

\put(-12, -50){\line(0,1){150}} \put(-12, -50){\line(1,1){150}}
\put(-12, -50){\circle*{2}}

\put(-28, -50){\line(0,1){150}} \put(-28, -50){\line(1,1){150}}
\put(-28, -50){\circle*{2}}

\put(-48, -70){\line(0,1){150}} \put(-48, -70){\line(1,1){150}}
\put(-48, -70){\circle*{2}}

\put(-90, -90){\line(0,1){170}} \put(-90, -90){\line(1,1){170}}
\put(-90, -90){\circle*{2}}

\put(10, 31){\line(0,1){70}} \put(10, 31){\line(1,1){70}} \put(10,
-10){\line(0,-1){80}} \put(10, -10){\line(-1,-1){80}} \put(10,
31){\circle*{2}} \put(17,30){$\nu$} \put(15,-12){$\theta\nu$}

\put(-23,70){$\alpha'$}
\put(-23,50){$\alpha$}
\put(-24,-10){$\omega'''$}
\put(-23,-30){$\omega''$}
\put(-23,-50){$\omega_n$}
\put(-43,-50){$\omega'$}
\put(-63,-70){$\omega_1$}



\end{picture}

\caption{The spectral sequence of a filtration for attaching a
single $(p,q)$-cell to a space with free cohomology.}
\label{fig:free1}
\end{figure}
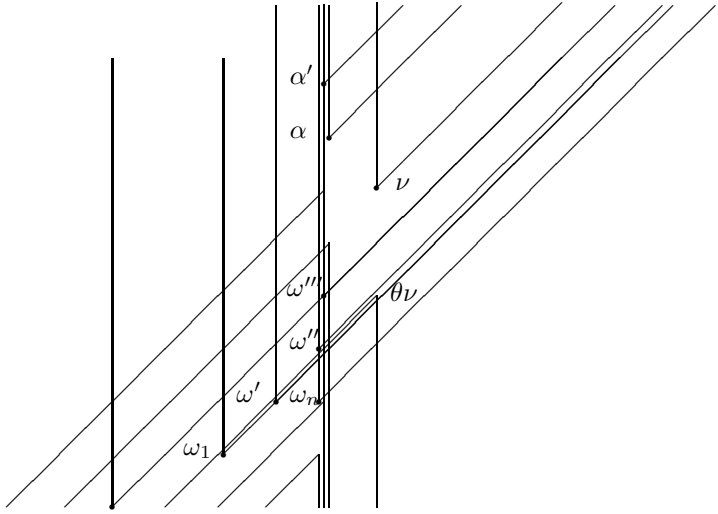

As before, a change of basis allows us to focus on a subset $\omega_1, \dots,
\omega_n$ of the free generators of $H^{*,*}(X^{(n)})$ whose
differentials hit the bottom cone of $\nu$ and that satisfy
\begin{enumerate}
\item $d(\omega_i) \neq 0$ for all $i$,
\item $|\omega_i| > |\omega_j|$ when $i > j$,
\item $|\omega_i^G| > |\omega_j^G|$ when $i > j$,
\end{enumerate}

\noindent and all other basis elements have zero differentials to
the bottom cone of $\nu$.  This is similar to what is referred to in
\cite{FL} as a ramp of length $n$.  Also, we can change the basis
again so that there is only one free generator, $\alpha$, of
$H^{*,*}(X^{(n)})$ with a nonzero differential to the top cone of
$\nu$.  Then, after this change of basis, the nonzero portion of the
spectral sequence of the filtration looks like the one in Figure
\ref{fig:free2}
\begin{figure}[htpb]

\centering

\begin{picture}(330,230)(-150,-110)

\put(-8, 50){\line(0,1){50}} \put(-8, 50){\line(1,1){50}} \put(-8,
10){\line(0,-1){100}} \put(-8, 10){\line(-1,-1){100}} \put(-8,
50){\circle*{2}} \put(-6,45){$\alpha$}

\put(-12, -50){\line(0,1){150}} \put(-12, -50){\line(1,1){150}}
\put(-12, -50){\circle*{2}} \put(-20,-60){$\omega_{n}$}

\put(-48, -70){\line(0,1){150}} \put(-48, -70){\line(1,1){150}}
\put(-48, -70){\circle*{2}} \put(-55,-80){$\omega_{n-1}$}

\put(-90, -90){\line(0,1){170}} \put(-90, -90){\line(1,1){170}}
\put(-90, -90){\circle*{2}}

\put(10, 31){\line(0,1){70}} \put(10, 31){\line(1,1){70}} \put(10,
-10){\line(0,-1){80}} \put(10, -10){\line(-1,-1){80}} \put(10,
31){\circle*{2}} \put(17,30){$\nu$} \put(15,-12){$\theta\nu$}



\end{picture}

\caption{The nonzero portion of the above spectral sequence, after a
change of basis.} \label{fig:free2}
\end{figure}
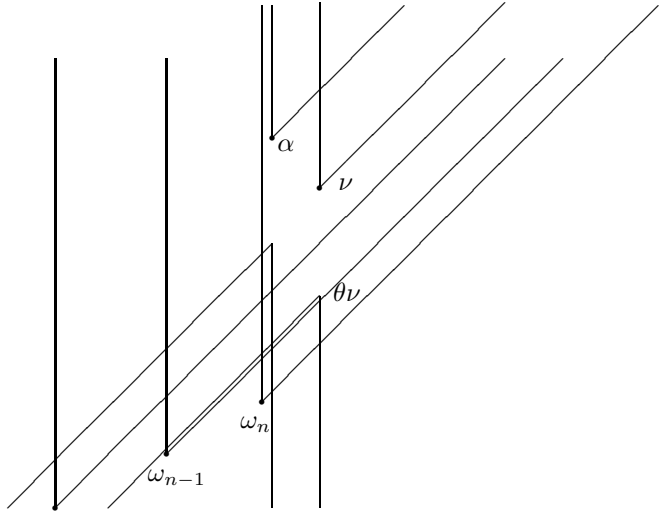

As above, $\alpha$ cannot support a nonzero differential, and we can see that each of the
$\omega_i$'s will shift up in $q$-degree and $\nu$ will shift down in $q$-degree.  That is, the $\omega_i$'s and $\nu$ each give rise to free generators in the cohomology of $H^{*,*}(X^{(n+1)})$, but in different bidegree than their predecessors. Thus, $H^{*,*}(X^{(n+1)})$ is again free.

Now suppose that $X^{(n+1)}$ is obtained by attaching a $(p,q)$-cell
$\nu'$ and that $X^{(n)}$ has a single $p$-cell $\nu$ already.  Then
by the previous case, the generator for $\nu$ was either shifted
down, killed off, or was left alone at the previous stage.  In any
case, because of our choice of ordering of the cells, the generator
for $\nu$ cannot support a differential to the generator for $\nu'$.
Thus, the only nonzero differentials to $\nu'$ are from strictly
lower dimensional cells.  Thus, we are reduced again to the previous
case and $H^{*,*}(X^{(n+1)})$ is free.  By induction, $H^{*,*}(X)$
is free. 

\end{proof}


\section{Real Projective Spaces and Grassmann Manifolds}
\label{sec:RPs}

In this section, $G=\Z/2$ exclusively, and the coefficient Mackey will always be $M = \underline{\Z/2}$ and will be suppressed from the notation.

Since each representation $\R^{p,q}$ has a linear $\Z/2$-action,
there is an induced action of $\Z/2$ on $G_n(\R^{p,q})$, the \bf
Grassmann manifold \rm of $n$-dimensional linear subspaces of
$\R^{p,q}$.  These Grassmann manifolds play a central role in the
classification of equivariant vector bundles, and so it is important
to understand their cohomology.  As a special case we have the real
projective spaces $\bP(\R^{p,q})=G_1(\R^{p,q})$.

The usual Schubert cell decomposition endows the Grassmann manifolds
with a $\text{Rep}(\Z/2)$-cell structure.  However, the number of
twists in each cell is dependent upon the flag of subrepresentations
of $\R^{p,q}$ that is chosen. A \bf flag symbol \rm $\varphi$ is a sequence
of integers $\varphi = (\varphi_1, \dots, \varphi_q)$ satisfying
$1\leq \varphi_1 < \dots < \varphi_q \leq q$.  A flag symbol $\varphi$ determines a flag of subrepresenations $V_0=0 \subset V_1 \subset \cdots
\subset V_p=\R^{p,q}$ satisfying
$V_{\varphi_i}/V_{\varphi_i-1}=\R^{1,1}$ for all $i=1, \dots, q$,
and all other quotients of consecutive terms are $\R^{1,0}$. For concreteness, we also require that $V_{i}$ is obtained from $V_{i-1}$ by adjoining a coordinate basis
vector.  For example, there is a
flag in $\R^{5,3}$ determined by the flag symbol $\varphi=(1,3,4)$
of the form $\R^{0,0} \subset \R^{1,1} \subset \R^{2,1} \subset
\R^{3,2} \subset \R^{4,3} \subset \R^{5,3}$.

A \bf Schubert symbol \rm $\sigma = (\sigma_1, \dots , \sigma_n)$ is a sequence of integers such that $1\leq\sigma_1 <
\sigma_2 < \dots < \sigma_n\leq p$. Given a Schubert symbol $\sigma$
and a flag symbol $\varphi$, let $e(\sigma,\varphi)$ be the set of
planes  $\ell \in G_n(\R^{p,q})$ for which $\dim(\ell \cap
V_{\sigma_i}) = 1+\dim(\ell \cap V_{\sigma_i-1})$, where $V_0
\subset \cdots \subset V_n$ is the flag determined by $\varphi$.
Then $e(\sigma,\varphi)$ is the interior of a cell $D(W)$ for some
representation $W$.  The dimension of the cell is determined by the
Schubert symbol $\sigma$ just as in nonequivariant topology, but the
number of twists depends on both $\sigma$ and the flag symbol
$\varphi$.

For example, consider $G_2(\R^{5,3})$, $\sigma = (3,5)$, and
$\varphi = (1,3,4)$.  Then $e(\sigma,\varphi)$ consists of planes
$\ell$ which have a basis with echelon form given by the matrix
below.

$$
\begin{array}{cc}
& \begin{array}{ccccc} \phantom{(}- & + & - & - & +\phantom{)}\end{array}\\
& \left( \begin{array}{ccccc} \ast & \ast & 1 & 0 & 0  \\
\ast & \ast & 0 & \ast & 1\end{array} \right)

\end{array}
$$

Here, the action of $\Z/2$ on the columns, as determined by
$\varphi$, has been indicated by inserting the appropriate signs
above the matrix.  After acting, this becomes the following.

$$
\begin{array}{cc}
& \begin{array}{ccccc} \phantom{(}- & + & - & - & +\phantom{)}\end{array}\\
& \left( \begin{array}{ccccc} -\ast & \ast & -1 & 0 & 0  \\
-\ast & \ast & 0 & -\ast & 1\end{array} \right)

\end{array}
$$

We require the last nonzero entry of each row to be 1, and so we
scale the fisrt row by $-1$.

$$
\begin{array}{cc}
& \begin{array}{ccccc} \phantom{(}- & + & - & - & +\phantom{)}\end{array}\\
& \left( \begin{array}{ccccc} \ast & -\ast & 1 & 0 & 0  \\
-\ast & \ast & 0 & -\ast & 1\end{array} \right)

\end{array}
$$

There are five coordinates which can be any real numbers, three
of which the $\Z/2$ action of multiplication by -1, so this is a
$(5,3)$-cell.  Through a similar process, we can obtain a cell
structure for $G_n(\R^{p,q})$ given any flag $\varphi$.  The type of
cell determined by the Schubert symbol $\sigma$ and the flag
$\varphi$ is given by the following proposition.  Here,
$\underline{\sigma_i} = \{1, \dots, \sigma_i\}$ and $\sigma(i) =
\{\sigma_1, \dots, \sigma_i\}$.

\begin{prop}
\label{prop:schubertcells} Let $\sigma = (\sigma_1, \dots,
\sigma_n)$ be a Schubert symbol and $\varphi = (\varphi_1, \dots,
\varphi_q)$ be a flag symbol for $\R^{p,q}$.  The cell
$e(\sigma,\varphi)$ of $G_n(\R^{p,q})$ is of dimension $(a,b)$ where
$a=\sum_{i=1}^n (\sigma_i-i)$ and $b=\sum_{\sigma_i \in \varphi}
|\underline{\sigma_i}\setminus (\varphi \cup
\sigma(i))|+\sum_{\sigma_i \notin \varphi} |(\underline{\sigma_i}
\cap \varphi)\setminus \sigma(i)|$.

\end{prop}
\begin{proof}
The formula for $a$ is exactly the same as in the nonequivariant
case.  The one for $b$ follows since the number of twisted
coordinates in each row is exactly the number of $\ast$ coordinates
for which the action is opposite to that on the coordinate
containing the 1 in that echelon row.
\end{proof}

\begin{cor} Real and complex projective spaces and Grassmann manifolds have free $RO(\Z/2)$-graded cohomology with $\underline{\Z/2}$ coefficients.
\end{cor}

\begin{prop} If $V \subseteq V'$ is an inclusion of representations and $\varphi \subseteq \varphi'$ is an extension of flag symbols for $V$ and $V'$, then there is a cellular inclusion $G_n(V) \inc G_n(V')$.

\end{prop}

The following theorem guarantees that the cohomology of Grassmann manifolds have cohomology generators in bijective correspondence with the Schubert cells.

\begin{thm} $H^{*,*}(G_n(\R^{u,v}))$ is a free $H^{*,*}(pt)$-module with generators in bijective correspondence with the Schubert cells. \label{thm:cellbij}
\end{thm}
\begin{proof}
Since $G_n(\R^{u,v})$ has a $\text{Rep}(\Z/2)$-complex structure, we know $H^{*,*}(G_n(\R^{u,v}))$ is free by the freeness theorem, Theorem \ref{thm:freeness}.  Let $\{\omega_1, \dots,
\omega_k\}$ be a set of free generators. Then $k \leq m$ where $m$
is the number of Schubert cells.

These spaces are based, so we can appeal to the forgetful long exact
sequence Lemma \ref{lemma:forget}.  By freeness and finite dimensionality,
the multiplication by $\rho$ map
is an injection for large enough $q$. Thus the forgetful map to
non-equivariant cohomology is surjective. Since
$H^{*}_{sing}(G_n(\R^{u,v}))$ is free with generators $a_1, \dots
a_m$ in bijective correspondence with the Schubert cells,
$H^{*,*}(G_n(\R^{u,v}))$ has a set of elements, $\{\alpha_1, \dots,
\alpha_m\}$, with $\psi(\alpha_i)=a_i$.  We can uniquely express
each $\alpha_i$ as $\alpha_i =\sum_{j=1}^k \rho^{e_{ij}}
\tau^{f_{ij}}\omega_j$.  We can ignore any terms that have $\rho$ in
them since $\psi(\rho)=0$. This gives a new set of elements,
$\bar{\alpha}_i = \sum_{j=1}^k \epsilon_{ij} \tau^{f_{ij}}\omega_j$,
where $\epsilon_{ij}=0$ or $1$ and $\psi(\bar{\alpha}_i)=a_i$. Since
$\psi(\tau)=1$, we have that $\sum_{j=1}^k
\epsilon_{ij}\psi(\omega_j)=a_i$.  Since linear combinations of the
linearly independent $\omega_j$'s map to the linearly independent
$a_i$'s, there are at least as many $\omega_j$'s as there are
$a_i$'s.  That is, $k \geq m$. 
\end{proof}

The above theorem is enough to determine the additive structure of the $RO(\Z/2)$-graded cohomology of the real projective spaces.

Recall that $\cU=(\R^{2,1})^\infty$ is a complete universe in the sense of \cite{Alaska}. Denote by
$\R\bP^\infty_{tw}=\bP(\cU)$, the space of lines in the
complete universe $\cU$.

Denote by $\R\bP^{n}_{tw}=\bP(\R^{n+1,\left\lfloor \frac{n+1}{2}
\right\rfloor})$, the equivariant space of lines in
$\R^{n+1,\left\lfloor \frac{n+1}{2} \right\rfloor}$.  For example,
$\R\bP^3_{tw} = \bP(\R^{4,2})$, $\R\bP^4_{tw} = \bP(\R^{5,2})$, and
$\R\bP^1_{tw} = S^{1,1}$.  There are natural cellular inclusions
$\R\bP^n_{tw} \inc \R\bP^{n+1}_{tw}$, the colimit of which is $\R\bP^\infty_{tw}$.

\begin{lem} $\R \bP^{n}_{tw}$ has a $\text{Rep}(\Z/2)$-structure with cells in dimension $(0,0)$, $(1,1)$, $(2,1)$, $(3,2)$, $(4,2)$, $\dots, (n,\left\lceil \frac{n}{2} \right\rceil)$.
\end{lem}
\begin{proof} This follows from Proposition \ref{prop:schubertcells} using the flag symbol $\varphi=(2, 4, 6, \dots)$. 
\end{proof}

\begin{lem}
$\R \bP^\infty_{tw}$ has a cell
structure with a single cell in dimension $(n, \left\lceil
\frac{n}{2} \right\rceil)$, for all $n \in \N$.
\end{lem}
\begin{proof}
The inclusions $\R\bP^{1}_{tw} \inc \R\bP^{2}_{tw} \inc
\cdots$ are cellular and their colimit is $\R\bP^\infty_{tw}$. 
\end{proof}

\begin{prop}

As a $H^{*,*}(pt)$-module, $H^{*,*}(\R \bP^{n}_{tw})$ is free with a
single generator in each degree $(k, \left\lceil \frac{k}{2}
\right\rceil)$ for $k= 0, 1, \dots, n$.

\end{prop}
\begin{proof}
Any nonzero differentials in the cellular spectral sequence associated to the cell structure using the flag symbol $\varphi = (2,4,6,\dots)$ would decrease the number of cohomology generators below the number of cells.  (See Figures \ref{fig:nodd} and \ref{fig:neven}.)  By Theorem \ref{thm:cellbij} this is not the case, and so the cohomology generators have degrees matching the dimensions of the cells. 
\end{proof}

\begin{figure}[htbp]
\centering
\begin{picture}(330,230)(-150,-110)

\put(-110,-50){\vector(1,0){235}}
\put(-50,-100){\vector(0,1){200}}

\put(-50, -51){\line(0,1){130}} \put(-50, -51){\line(1,1){130}}
\put(-50, -91){\line(0,-1){20}} \put(-50, -91){\line(-1,-1){20}}
\put(-50, -51){\circle*{2}}

\put(-28, -31){\line(0,1){110}} \put(-28, -31){\line(1,1){110}}
\put(-28, -71){\line(0,-1){40}} \put(-28, -71){\line(-1,-1){40}}
\put(-28, -31){\circle*{2}} \put(-26,-35){$a_{1,1}$}

\put(-10, -31){\line(0,1){110}} \put(-10, -31){\line(1,1){110}}
\put(-10, -71){\line(0,-1){40}} \put(-10, -71){\line(-1,-1){40}}
\put(-10, -31){\circle*{2}} \put(-6,-35){$a_{2,1}$}

\put(12, -11){\line(0,1){90}} \put(12, -11){\line(1,1){90}} \put(12,
-51){\line(0,-1){60}} \put(12, -51){\line(-1,-1){60}} \put(12,
-11){\circle*{2}}

\put(32, -11){\line(0,1){90}} \put(32, -11){\line(1,1){70}} \put(32,
-51){\line(0,-1){60}} \put(32, -51){\line(-1,-1){60}} \put(32,
-11){\circle*{2}}

\put(52, 10){\line(0,1){70}} \put(52, 10){\line(1,1){50}} \put(52,
-30){\line(0,-1){80}} \put(52, -30){\line(-1,-1){80}} \put(52,
10){\circle*{2}} \put(55,7){$a_{n,(n-1)/2}$}

\multiput(-90,-51)(20,0){11}{\line(0,1){3}}
\multiput(-52,-90)(0,20){10}{\line(1,0){3}}
\put(-46,-57){$\scriptscriptstyle{0}$}
\put(-26,-57){$\scriptscriptstyle{1}$}
\put(-6,-57){$\scriptscriptstyle{2}$}
\put(14,-57){$\scriptscriptstyle{3}$}
\put(34,-57){$\scriptscriptstyle{4}$}
\put(-66,-57){$\scriptscriptstyle{-1}$}
\put(-86,-57){$\scriptscriptstyle{-2}$}
\put(-57,-47){$\scriptscriptstyle{0}$}
\put(-57,-27){$\scriptscriptstyle{1}$}
\put(-57,-7){$\scriptscriptstyle{2}$}
\put(-57,13){$\scriptscriptstyle{3}$}

\put(-58,95){${q}$} \put(115,-60){${p}$}
\end{picture}

\caption{The $E_1$ page of the cellular spectral sequence for
$\R\bP^{n}_{tw}$ for $n$ odd.} \label{fig:nodd}
\end{figure}
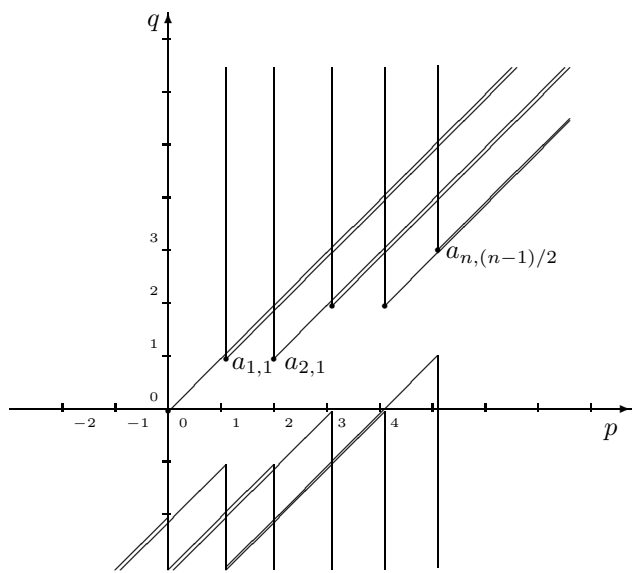

\begin{figure}[htbp]
\centering
\begin{picture}(330,230)(-150,-110)
\put(-110,-50){\vector(1,0){235}}
\put(-50,-100){\vector(0,1){200}}

\put(-50, -51){\line(0,1){130}} \put(-50, -51){\line(1,1){130}}
\put(-50, -91){\line(0,-1){20}} \put(-50, -91){\line(-1,-1){20}}
\put(-50, -51){\circle*{2}}

\put(-28, -31){\line(0,1){110}} \put(-28, -31){\line(1,1){110}}
\put(-28, -71){\line(0,-1){40}} \put(-28, -71){\line(-1,-1){40}}
\put(-28, -31){\circle*{2}} \put(-26,-35){$a_{1,1}$}

\put(-10, -31){\line(0,1){110}} \put(-10, -31){\line(1,1){110}}
\put(-10, -71){\line(0,-1){40}} \put(-10, -71){\line(-1,-1){40}}
\put(-10, -31){\circle*{2}} \put(-6,-35){$a_{2,1}$}

\put(12, -11){\line(0,1){90}} \put(12, -11){\line(1,1){90}} \put(12,
-51){\line(0,-1){60}} \put(12, -51){\line(-1,-1){60}} \put(12,
-11){\circle*{2}}

\put(32, -11){\line(0,1){90}} \put(32, -11){\line(1,1){70}} \put(32,
-51){\line(0,-1){60}} \put(32, -51){\line(-1,-1){60}} \put(32,
-11){\circle*{2}}

\put(52, 10){\line(0,1){70}} \put(52, 10){\line(1,1){50}} \put(52,
-30){\line(0,-1){80}} \put(52, -30){\line(-1,-1){80}} \put(52,
10){\circle*{2}}

\put(70, 10){\line(0,1){70}} \put(70, 10){\line(1,1){30}} \put(70,
-30){\line(0,-1){80}} \put(70, -30){\line(-1,-1){80}} \put(70,
10){\circle*{2}} \put(73,7){$a_{n,n/2}$}

\multiput(-90,-51)(20,0){11}{\line(0,1){3}}
\multiput(-52,-90)(0,20){10}{\line(1,0){3}}
\put(-46,-57){$\scriptscriptstyle{0}$}
\put(-26,-57){$\scriptscriptstyle{1}$}
\put(-6,-57){$\scriptscriptstyle{2}$}
\put(14,-57){$\scriptscriptstyle{3}$}
\put(34,-57){$\scriptscriptstyle{4}$}
\put(-66,-57){$\scriptscriptstyle{-1}$}
\put(-86,-57){$\scriptscriptstyle{-2}$}
\put(-57,-47){$\scriptscriptstyle{0}$}
\put(-57,-27){$\scriptscriptstyle{1}$}
\put(-57,-7){$\scriptscriptstyle{2}$}
\put(-57,13){$\scriptscriptstyle{3}$}

\put(-58,95){${q}$} \put(115,-60){${p}$}
\end{picture}
\caption{The $E_1$ page of the cellular spectral sequence for
$\R\bP^{n}_{tw}$ for $n$ even.} \label{fig:neven}
\end{figure}
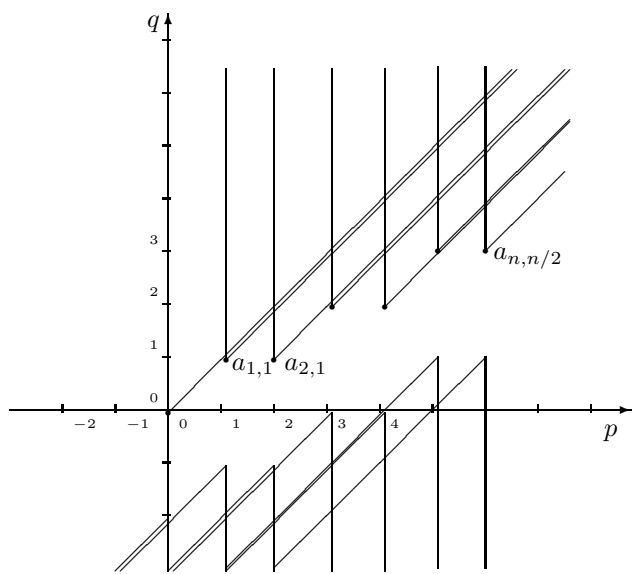

\begin{prop}

As a $H^{*,*}(pt)$-module, $H^{*,*}(\R \bP^\infty_{tw})$ is free
with a single generator in each degree $(n, \left\lceil \frac{n}{2}
\right\rceil)$, for all $n \in \N$.

\end{prop}

\begin{proof}

$\R\bP^\infty_{tw}$ is the colimit of the above projective spaces.
Thus, any non-zero differential for $\R\bP^\infty_{tw}$ would induce
a non-zero differential at some finite stage.  By the above proposition, this is not the case. 

\end{proof}

\begin{lem}

As a $H^{*,*}(pt)$-module, $H^{*,*}(S^{1,1})$ is free with a single
generator $a$ in degree $(1,1)$.  As a ring, $H^{*,*}(S^{1,1}) \cong
H^{*,*}(pt)[a]/(a^2 = \rho a)$.

\end{lem}
\begin{proof}

The statement about the module structure is immediate since $S^{1,1}
\cong \R\bP^1_{tw}$.

Since $S^{1,1}$ is a $K(\Z(1),1)$, we can consider $a \in
[S^{1,1},S^{1,1}]$ as the class of the identity and $\rho \in [pt,
S^{1,1}]$ as the inclusion. Then $a^2$ is the composite

$\xymatrix{a^2 \colon S^{1,1} \ar[r]^(0.45)\Delta & S^{1,1} \Smash
S^{1,1} \ar[r]^(0.6){a \Smash a} & S^{2,2} \ar[r] & K(\Z/2(2),2)}$.

\noindent Similarly, $\rho a$ is the composite

$\xymatrix{\rho a \colon S^{1,1} \ar[r] & S^{0,0} \Smash S^{1,1}
\ar[r]^(0.6){\rho \Smash a} & S^{2,2} \ar[r] & K(\Z/2(2),2)}$.

\noindent The claim is that these two maps are homotopic.
Considering the spheres involved as one point compactifications of
the corresponding representations, the map $a^2$ is inclusion of
$(\R^{1,1})^+$ as the diagonal in $(\R^{2,2})^+$ and $\rho a$ is
inclusion of $(\R^{1,1})^+$ as the vertical axis.  There is then an
equivariant homotopy $H\colon (\R^{1,1})^+ \times I \ra
(\R^{2,2})^+$ between these two maps given by $H(x,t) = (t x, x)$.

\end{proof}

From here, we are poised to compute the ring structure of the $RO(\Z/2)$-graded cohomology of each real projective space.

\begin{thm}
\label{thm:rpinfty} $H^{*,*}(\R\bP^\infty_{tw})=H^{*,*}(pt)[a,
b]/(a^2=\rho a +\tau b)$, where $\deg(a)=(1,1)$ and $\deg(b)=(2,1)$.
\end{thm}
\begin{proof}
It remains to compute the multiplicative
structure of the cohomology ring. Denote by $a = a_{(1, 1)}$, and $b=a_{(2,1)}$.
By Lemma \ref{lemma:forget}, the forgetful map $\psi\colon  H^{*,*}(\R
\bP^\infty_{tw}) \ra
H_{sing}^*(\R\bP^\infty)$ maps $\psi(a) = z$ and $\psi(b) = z^2$
where $z \in H^1_{sing}(\R\bP^\infty)$ is the ring generator for
non-equivariant cohomology.   Since $\psi$ is a homomorphism of rings,
$\psi(ab)=z^3 \neq 0$, and so the product $ab$ is nonzero in $H^{*,*}(\R
\bP^\infty_{tw})$.
Observe that $\rho b$ is also in degree $(3,2)$ in $H^{*,*}(\R
\bP^\infty_{tw})$, but
$\psi(\rho b) = 0$ since $\psi(\rho)=0$.  Thus $ab$ and $\rho b$
generate $H^{*,*}(\R
\bP^\infty_{tw})$ in degree $(3, 2)$.  Also, $\psi(b^2)=z^4$, and so
$b^2$ in nonzero in $H^{*,*}(\R
\bP^\infty_{tw})$.  This means that $b^2$ is the unique
nonzero element of $H^{*,*}(\R
\bP^\infty_{tw})$ in degree $(4,2)$.  Inductively, it can be
shown that if $n$ is even the unique nonzero element of $R$ in
degree $(n,\frac{n}{2})$ is $b^{n/2}$ and that if $n$ is odd, then
$ab^{(n-1)/2}$ is linearly independent from $\rho b^{(n-1)/2}$.

Now, $a^2 \in H^{2,2}(\R \bP^\infty_{tw})$ and so is a linear
combination of $\rho a$ and $\tau b$.  Since $\psi(a^2)=z^2$, there
must be a $\tau b$ term in the expression for $a^2$.  Also, upon
restriction to $\R\bP^1_{tw}=S^{1,1}$, $a^2$ restricts to $a^2=\rho
a$.  Thus, $a^2=\rho a + \tau b \in H^{*,*}(\R
\bP^\infty_{tw})$.

\end{proof}

\begin{thm} Let $n > 2$.  If $n$ is even, then $H^{*,*}(\bP(\R^{n,\frac{n}{2}})) = H^{*,*}(pt)[a_{1,1},b_{2,1}]/ \sim$ where the generating relations are $a^2 = \rho a + \tau b$ and $b^k =0$ for $k \geq \frac{n}{2}$.  If $n$ is odd, then $H^{*,*}(\bP(\R^{n,\frac{n-1}{2}})) = H^{*,*}(pt)[a_{1,1},b_{2,1}]/ \sim$ where the generating relations are $a^2 = \rho a + \tau b$,  $b^k =0$ for $k \geq \frac{n+1}{2}$, and $a\cdot b^{(n-1)/2}=0$.

\end{thm}

\begin{proof}
Only the multiplicative structure needs to be checked since the
cohomology is free and the generators given above are in the correct
degrees.  Considering the restriction of the corresponding
classes $a$ and $b$ in $H^{*,*}(\R\bP^\infty_{tw})$, the relation
$a^2 = \rho a + \tau b$ is immediate. The relations $b^k =0$ for $k
> \frac{n}{2}$ when $n$ is even and $b^k =0$ for $k \geq
\frac{n+1}{2}$ when $n$ is odd follow for degree reasons. Also,
since the class $ab^{(n-1)/2} \in H^{*,*}(\R\bP^\infty_{tw})$ is a free
generator, it restricts to zero in
$H^{*,*}(\bP(\R^{n,\frac{n-1}{2}}))$.  Thus $ab^{(n-1)/2} = 0 \in
H^{*,*}(\bP(\R^{n,\frac{n-1}{2}}))$.

\end{proof}

We can also compute the cohomology of projective spaces associated
to arbitrary representations.  The following easy lemma will be
useful.  In particular, it allows us to only consider the projective
spaces associated to representations $V \cong \R^{p,q}$ where $q\leq
p/2$.

\begin{lem}
\label{lemma:easyRP} $\bP(\R^{p,q}) \cong \bP(\R^{p,p-q})$.

\end{lem}
\begin{proof}
Consider a basis of $\R^{p,q}$ in which the first $q$ coordinates
have the nontrivial action, and a basis of $\R^{p,p-q}$ in which the
first $q$ coordinates are fixed by the action.  Then the map $f
\colon \bP(\R^{p,q}) \ra \bP(\R^{p,p-q})$ that sends the span of
$(x_1, \dots , x_p)$ to the span of $(x_1, \dots , x_p)$ is
equivariant.  It is clearly a homeomorphism. 
\end{proof}

\begin{lem}If $q \leq p/2$, then $\bP(\R^{p,q})$ has a cell structure with a single cell in each dimension $(0,0)$, $(1,1)$, $(2,1)$, $(3,2)$, $(4,2), \dots,$ $(2q-1,q)$, $(2q,q)$, $\dots,$ $(p-1,q)$.
\end{lem}

For example, $\bP(\R^{4,1})$ has a single cell in each dimension
$(0,0)$, $(1,1)$, $(2,1)$, and $(3,1)$.

\begin{proof}
The result follows by Proposition \ref{prop:schubertcells} using the flag symbol $\varphi=(2,4,\dots,2q)$. 

\end{proof}

\begin{lem}
\label{lemma:RPmodule} As a $H^{*,*}(pt)$-module,
$H^{*,*}(\bP(\R^{p,q}))$ is free with a single generator in
degrees $(0,0)$, $(1,1)$, $(2,1)$, $(3,2)$, $(4,2), \dots,$
$(2q,q)$, $(2q+1,q), \dots,$ $(p-1,q)$.

\end{lem}

\begin{proof}

Using the cell structure in the previous lemma, Theorem \ref{thm:cellbij} implies there can be no nonzero differentials in the cellular spectral sequence.

\end{proof}

The ring structure of the other projective spaces can be computed by considering the restriction of
$H^{*,*}(\R\bP^\infty_{tw})$ to $H^{*,*}(\bP(\R^{p,q}))$.
\begin{prop}
\label{prop:ringRP} $H^{*,*}(\bP(\R^{p,q}))$ is a truncated
polynomial algebra over $H^{*,*}(pt)$ on generators in degrees
$(1,1)$, $(2,1)$, $(2q+1,q)$, $(2q+2,q), \dots,$ $(p-1,q)$, subject
to the relations determined by the restriction of

$H^{*,*}(\R\bP^\infty_{tw})$ to $H^{*,*}(\bP(\R^{p,q}))$.

\end{prop}

For example, consider $\bP(\R^{4,1})$.  By the above proposition, $H^{*,*}(\bP(\R^{4,1}))$ is generated by classes
$a_{1,1}$, $b_{2,1}$, and $c_{3,1}$.  The classes $a$
and $b$ in $H^{*,*}(\R\bP^\infty_{tw})$ restrict to $a$ and $b$
respectively, so $a^2 = \rho a + \tau b$ in
$H^{*,*}(\bP(\R^{4,1}))$.  Now, $ab$ has degree $(3,2)$ and so $ab =
? \rho b + ? \tau c$.  However, the product $ab$ in
$H^{*,*}(\R\bP^\infty_{tw})$ restricts to the class $\tau c$.  Since
restriction is a map of rings, $ab = \tau c$ in
$H^{*,*}(\bP(\R^{4,1}))$.  Similar considerations show that $bc =0$
and $c^2=0$.  Thus $H^{*,*}(\bP(\R^{4,1}))=H^{*,*}(pt)[a_{1,1},
b_{2, 1}, c_{3,1}]/\sim$, where the generating relations are $a^2 =
\rho a + \tau b$, $ab = \tau c$, $bc=0$, and $c^2=0$.

In some cases, the Freeness Theorem is enough to determine the
additive structure of the $RO(\Z/2)$-graded cohomology of Grassmann
manifolds.

\begin{prop} $G_2(\R^{p,1})$ has a $\text{Rep}(\Z/2)$-complex structure so that
$H^{*,*}(G_n(\R^{p,1}))$ is a free $H^{*,*}(pt)$-module on
generators whose degree match the dimensions of the cells.
\end{prop}

\begin{proof}
Using the flag symbol $\varphi=(2)$, every cell, except the
$(0,0)$-cell, has either one or two twists. The cells are in
bidegrees so that there can be no dimension shifting in the cellular
spectral sequence. The result now follows by Theorem
\ref{thm:cellbij}. 
\end{proof}

For example, $H^{*,*}(G_2(\R^{4,1});\underline{\Z/2})$ is a free
$H^{*,*}(pt;\underline{\Z/2})$-module with generators in degrees
$(0,0)$, $(1,1)$, $(2,1)$, $(2,1)$, $(3,1)$, and $(4,2)$ (see Figure
\ref{fig:g2r41}).

\begin{figure}[htbp]

\centering

\begin{picture}(330,230)(-150,-110)
\put(-110,-50){\vector(1,0){235}}
\put(-50,-100){\vector(0,1){200}}

\put(-50, -51){\line(0,1){70}} \put(-50, -51){\line(1,1){150}}
\put(-50, -91){\line(0,-1){20}} \put(-50, -91){\line(-1,-1){20}}
\put(-50, -51){\circle*{2}}

\put(-28, -31){\line(0,1){130}} \put(-28, -31){\line(1,1){130}}
\put(-28, -71){\line(0,-1){40}} \put(-28, -71){\line(-1,-1){40}}
\put(-28, -31){\circle*{2}}

\put(-10, -31){\line(0,1){130}} \put(-10, -31){\line(1,1){110}}
\put(-10, -71){\line(0,-1){40}} \put(-10, -71){\line(-1,-1){40}}
\put(-10, -31){\circle*{2}}

\put(-8, -31){\line(0,1){130}} \put(-8, -31){\line(1,1){110}}
\put(-8, -71){\line(0,-1){40}} \put(-8, -71){\line(-1,-1){40}}
\put(-8, -31){\circle*{2}}

\put(10, -31){\line(0,1){130}} \put(10, -31){\line(1,1){90}}
\put(10, -71){\line(0,-1){40}} \put(10, -71){\line(-1,-1){40}}
\put(10, -31){\circle*{2}}

\put(32, -11){\line(0,1){110}} \put(32, -11){\line(1,1){70}}
\put(32, -51){\line(0,-1){60}} \put(32, -51){\line(-1,-1){60}}
\put(32, -11){\circle*{2}}

\multiput(-90,-51)(20,0){11}{\line(0,1){3}}
\multiput(-52,-90)(0,20){10}{\line(1,0){3}}
\put(-46,-57){$\scriptscriptstyle{0}$}
\put(-26,-57){$\scriptscriptstyle{1}$}
\put(-6,-57){$\scriptscriptstyle{2}$}
\put(14,-57){$\scriptscriptstyle{3}$}
\put(34,-57){$\scriptscriptstyle{4}$}
\put(-66,-57){$\scriptscriptstyle{-1}$}
\put(-86,-57){$\scriptscriptstyle{-2}$}
\put(-57,-47){$\scriptscriptstyle{0}$}
\put(-57,-27){$\scriptscriptstyle{1}$}
\put(-57,-7){$\scriptscriptstyle{2}$}
\put(-57,13){$\scriptscriptstyle{3}$}

%


\put(-58,95){${q}$} \put(125,-60){${p}$}
\end{picture}
\caption{$H^{*,*}(G_2(\R^{4,1}))$} \label{fig:g2r41}
\end{figure}

Interestingly, there are situations where there must be nonzero
differentials in the cellular spectral sequences.

As another example, consider now $X=G_2(\R^{4,2})$.  Consider the
three flag symbols $\varphi_1=(2,3)$, $\varphi_2=(2,4)$, and
$\varphi_3=(3,4)$. The spectral sequences associated to the cell
structures with these flag symbols have $E_1$ term given in Figures
\ref{fig:g2r42phi23}, \ref{fig:g2r42phi24}, and \ref{fig:g2r42phi34}
repsectively.

\begin{figure}[htpb]

\centering

\begin{picture}(330,230)(-150,-110)
\put(-110,-50){\vector(1,0){235}}
\put(-50,-100){\vector(0,1){200}}

\put(-50, -51){\line(0,1){130}} \put(-50, -51){\line(1,1){130}}
\put(-50, -91){\line(0,-1){20}} \put(-50, -91){\line(-1,-1){20}}
\put(-50, -51){\circle*{2}}

\put(-28, -51){\line(0,1){130}} \put(-28, -51){\line(1,1){130}}
\put(-28, -91){\line(0,-1){20}} \put(-28, -91){\line(-1,-1){20}}
\put(-28, -51){\circle*{2}}

\put(-10, -11){\line(0,1){90}} \put(-10, -11){\line(1,1){90}}
\put(-10, -51){\line(0,-1){60}} \put(-10, -51){\line(-1,-1){60}}
\put(-10, -11){\circle*{2}}

\put(-8, -11){\line(0,1){90}} \put(-8, -11){\line(1,1){90}} \put(-8,
-51){\line(0,-1){60}} \put(-8, -51){\line(-1,-1){60}} \put(-8,
-11){\circle*{2}}

\put(10, -11){\line(0,1){90}} \put(10, -11){\line(1,1){90}} \put(10,
-51){\line(0,-1){60}} \put(10, -51){\line(-1,-1){60}} \put(10,
-11){\circle*{2}}

\put(32, -11){\line(0,1){90}} \put(32, -11){\line(1,1){90}} \put(32,
-51){\line(0,-1){60}} \put(32, -51){\line(-1,-1){60}} \put(32,
-11){\circle*{2}}

\multiput(-90,-51)(20,0){11}{\line(0,1){3}}
\multiput(-52,-90)(0,20){10}{\line(1,0){3}}
\put(-46,-57){$\scriptscriptstyle{0}$}
\put(-26,-57){$\scriptscriptstyle{1}$}
\put(-6,-57){$\scriptscriptstyle{2}$}
\put(14,-57){$\scriptscriptstyle{3}$}
\put(34,-57){$\scriptscriptstyle{4}$}
\put(-66,-57){$\scriptscriptstyle{-1}$}
\put(-86,-57){$\scriptscriptstyle{-2}$}
\put(-57,-47){$\scriptscriptstyle{0}$}
\put(-57,-27){$\scriptscriptstyle{1}$}
\put(-57,-7){$\scriptscriptstyle{2}$}
\put(-57,13){$\scriptscriptstyle{3}$}
%


\put(-58,95){${q}$} \put(125,-60){${p}$}
\end{picture}

\caption{The $E_1$ page of the cellular spectral sequence for
$G_2(\R^{4,2})$ using $\varphi_1=(2,3)$.} \label{fig:g2r42phi23}

\end{figure}
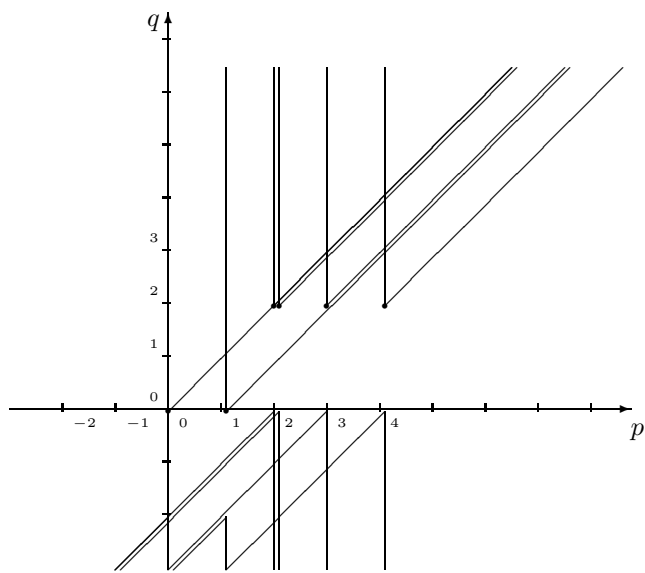

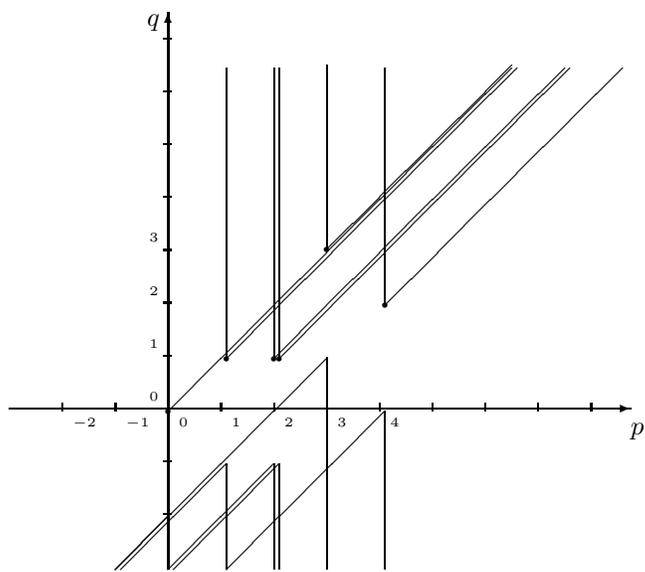
\begin{figure}[htpb]

\centering

\begin{picture}(330,230)(-150,-110)
\put(-110,-50){\vector(1,0){235}}
\put(-50,-100){\vector(0,1){200}}

\put(-50, -51){\line(0,1){130}} \put(-50, -51){\line(1,1){130}}
\put(-50, -91){\line(0,-1){20}} \put(-50, -91){\line(-1,-1){20}}
\put(-50, -51){\circle*{2}}

\put(-28, -31){\line(0,1){110}} \put(-28, -31){\line(1,1){110}}
\put(-28, -71){\line(0,-1){40}} \put(-28, -71){\line(-1,-1){40}}
\put(-28, -31){\circle*{2}}

\put(-10, -31){\line(0,1){110}} \put(-10, -31){\line(1,1){110}}
\put(-10, -71){\line(0,-1){40}} \put(-10, -71){\line(-1,-1){40}}
\put(-10, -31){\circle*{2}}

\put(-8, -31){\line(0,1){110}} \put(-8, -31){\line(1,1){110}}
\put(-8, -71){\line(0,-1){40}} \put(-8, -71){\line(-1,-1){40}}
\put(-8, -31){\circle*{2}}

\put(10, 10){\line(0,1){70}} \put(10, 10){\line(1,1){70}} \put(10,
-31){\line(0,-1){80}} \put(10, -31){\line(-1,-1){80}} \put(10,
10){\circle*{2}}

\put(32, -11){\line(0,1){90}} \put(32, -11){\line(1,1){90}} \put(32,
-51){\line(0,-1){60}} \put(32, -51){\line(-1,-1){60}} \put(32,
-11){\circle*{2}}

\multiput(-90,-51)(20,0){11}{\line(0,1){3}}
\multiput(-52,-90)(0,20){10}{\line(1,0){3}}
\put(-46,-57){$\scriptscriptstyle{0}$}
\put(-26,-57){$\scriptscriptstyle{1}$}
\put(-6,-57){$\scriptscriptstyle{2}$}
\put(14,-57){$\scriptscriptstyle{3}$}
\put(34,-57){$\scriptscriptstyle{4}$}
\put(-66,-57){$\scriptscriptstyle{-1}$}
\put(-86,-57){$\scriptscriptstyle{-2}$}
\put(-57,-47){$\scriptscriptstyle{0}$}
\put(-57,-27){$\scriptscriptstyle{1}$}
\put(-57,-7){$\scriptscriptstyle{2}$}
\put(-57,13){$\scriptscriptstyle{3}$}
%


\put(-58,95){${q}$} \put(125,-60){${p}$}
\end{picture}
\caption{The $E_1$ page of the cellular spectral sequence for
$G_2(\R^{4,2})$ using $\varphi_2=(2,4)$.} \label{fig:g2r42phi24}

\end{figure}

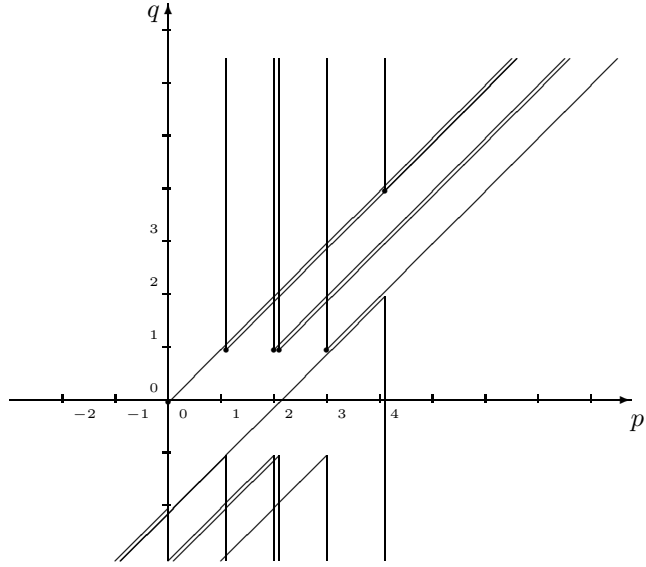
\begin{figure}[htpb]

\centering

\begin{picture}(330,230)(-150,-110)
\put(-110,-50){\vector(1,0){235}}
\put(-50,-100){\vector(0,1){200}}

\put(-50, -51){\line(0,1){130}} \put(-50, -51){\line(1,1){130}}
\put(-50, -91){\line(0,-1){20}} \put(-50, -91){\line(-1,-1){20}}
\put(-50, -51){\circle*{2}}

\put(-28, -31){\line(0,1){110}} \put(-28, -31){\line(1,1){110}}
\put(-28, -71){\line(0,-1){40}} \put(-28, -71){\line(-1,-1){40}}
\put(-28, -31){\circle*{2}}

\put(-10, -31){\line(0,1){110}} \put(-10, -31){\line(1,1){110}}
\put(-10, -71){\line(0,-1){40}} \put(-10, -71){\line(-1,-1){40}}
\put(-10, -31){\circle*{2}}

\put(-8, -31){\line(0,1){110}} \put(-8, -31){\line(1,1){110}}
\put(-8, -71){\line(0,-1){40}} \put(-8, -71){\line(-1,-1){40}}
\put(-8, -31){\circle*{2}}

\put(10, -31){\line(0,1){110}} \put(10, -31){\line(1,1){110}}
\put(10, -71){\line(0,-1){40}} \put(10, -71){\line(-1,-1){40}}
\put(10, -31){\circle*{2}}

\put(32, 29){\line(0,1){50}} \put(32, 29){\line(1,1){50}} \put(32,
-11){\line(0,-1){100}} \put(32, -11){\line(-1,-1){100}} \put(32,
29){\circle*{2}}

\multiput(-90,-51)(20,0){11}{\line(0,1){3}}
\multiput(-52,-90)(0,20){10}{\line(1,0){3}}
\put(-46,-57){$\scriptscriptstyle{0}$}
\put(-26,-57){$\scriptscriptstyle{1}$}
\put(-6,-57){$\scriptscriptstyle{2}$}
\put(14,-57){$\scriptscriptstyle{3}$}
\put(34,-57){$\scriptscriptstyle{4}$}
\put(-66,-57){$\scriptscriptstyle{-1}$}
\put(-86,-57){$\scriptscriptstyle{-2}$}
\put(-57,-47){$\scriptscriptstyle{0}$}
\put(-57,-27){$\scriptscriptstyle{1}$}
\put(-57,-7){$\scriptscriptstyle{2}$}
\put(-57,13){$\scriptscriptstyle{3}$}
%


\put(-58,95){${q}$} \put(125,-60){${p}$}
\end{picture}
\caption{The $E_1$ page of the cellular spectral sequence for
$G_2(\R^{4,2})$ using $\varphi_3=(3,4)$.} \label{fig:g2r42phi34}

\end{figure}

The cohomology of $X$ can be deduced by comparing these three
cellular spectral sequences. We can see from the picture for
$\varphi_2$ that $H^{1,0}(X)=0$, and so the differential leaving the
$(1,0)$ generator in the $\varphi_1$ spectral sequence is non-zero.
Thus, $H^{1,1}(X)=\Z/2$, $H^{2,1}(X)=\Z/2$ and $H^{2,0}(X)=\Z/2$.
In particular, there is a free generator in degree $(1,1)$ and there
is a nontrivial differential leaving the $(2,1)$ generators of the
spectral sequence for $\varphi_2$.  After a change of basis, if
necessary, the differential can be adjusted so that it is zero on
one of the $(2,1)$ generators and the other generator maps
nontrivially.  Now from $\varphi_1$ we see that $H^{4,1}(X)=0$, and
so there is a nontrivial differential leaving the $(3,1)$ generator
in the $\varphi_3$ spectral sequence.  This means that the $(4,2)$
generator in the $\varphi_1$ and $\varphi_2$ spectral sequences must
survive.  Thus, all differentials in the $\varphi_2$ spectral
sequence are known.  They are all zero, except for the one leaving
the two $(2,1)$ generators, which behaves as described above. That
spectral sequence collapses almost immediately to give the
cohomology of $G_2(\R^{4,2})$ pictured in Figure \ref{fig:g2r42}.

\begin{figure}[htpb]

\centering

\begin{picture}(330,230)(-150,-110)
\put(-110,-50){\vector(1,0){235}}
\put(-50,-100){\vector(0,1){200}}

\put(-50, -51){\line(0,1){130}} \put(-50, -51){\line(1,1){130}}
\put(-50, -91){\line(0,-1){20}} \put(-50, -91){\line(-1,-1){20}}
\put(-50, -51){\circle*{2}}

\put(-28, -31){\line(0,1){110}} \put(-28, -31){\line(1,1){110}}
\put(-28, -71){\line(0,-1){40}} \put(-28, -71){\line(-1,-1){40}}
\put(-28, -31){\circle*{2}}

\put(-10, -31){\line(0,1){110}} \put(-10, -31){\line(1,1){110}}
\put(-10, -71){\line(0,-1){40}} \put(-10, -71){\line(-1,-1){40}}
\put(-10, -31){\circle*{2}}

\put(-8, -13){\line(0,1){90}} \put(12, -11){\line(1,1){90}} \put(-8,
-71){\line(0,-1){40}} \put(-8, -71){\line(-1,-1){40}} \put(-8,
-13){\circle*{2}} \put(12, -11){\circle*{2}}

\put(10, 10){\line(0,1){70}} \put(10, 10){\line(1,1){70}} \put(10,
-51){\line(0,-1){60}} \put(-10, -51){\line(-1,-1){40}} \put(10,
10){\circle*{2}} \put(10, -51){\circle*{2}} \put(-10,
-51){\circle*{2}}

\put(32, -11){\line(0,1){90}} \put(32, -11){\line(1,1){70}} \put(32,
-51){\line(0,-1){60}} \put(32, -51){\line(-1,-1){60}} \put(32,
-11){\circle*{2}}

\multiput(-90,-51)(20,0){11}{\line(0,1){3}}
\multiput(-52,-90)(0,20){10}{\line(1,0){3}}
\put(-46,-57){$\scriptscriptstyle{0}$}
\put(-26,-57){$\scriptscriptstyle{1}$}
\put(-6,-57){$\scriptscriptstyle{2}$}
\put(14,-57){$\scriptscriptstyle{3}$}
\put(34,-57){$\scriptscriptstyle{4}$}
\put(-66,-57){$\scriptscriptstyle{-1}$}
\put(-86,-57){$\scriptscriptstyle{-2}$}
\put(-57,-47){$\scriptscriptstyle{0}$}
\put(-57,-27){$\scriptscriptstyle{1}$}
\put(-57,-7){$\scriptscriptstyle{2}$}
\put(-57,13){$\scriptscriptstyle{3}$}
%


\put(-58,95){${q}$} \put(125,-60){${p}$}
\end{picture}
\caption{$H^{*,*}(G_2(\R^{4,2}))$} \label{fig:g2r42}
\end{figure}

By the Freeness Theorem \ref{thm:freeness}, we know that
$H^{*,*}(G_2(\R^{4,2}))$ is free.   Counting the $\Z/2$ vector space
dimensions in each bidegree reveals that the degrees are the same as
those of a free $H^{*,*}(pt)$-module with generators in degrees
$(1,1)$, $(2,1)$, $(2,2)$, ($3,2)$, and $(4,2)$. This is the only
free $H^{*,*}(pt)$-module with these $\Z/2$ dimensions, and so we
have the following computation.

\begin{prop}$H^{*,*}(G_2(\R^{4,2}))$ is a free $H^{*,*}(pt)$-module with generators in degrees $(1,1)$, $(2,1)$, $(2,2)$, $(3,2)$, and $(4,2)$.
\end{prop}

That is, $H^{*,*}(G_2(\R^{4,2}))$ has free generators as displayed
in Figure \ref{fig:cohomg2r42}.

\begin{figure}[htpb]

\centering

\begin{picture}(330,230)(-150,-110)
\put(-110,-50){\vector(1,0){235}}
\put(-50,-100){\vector(0,1){200}}

\put(-50, -51){\line(0,1){130}} \put(-50, -51){\line(1,1){130}}
\put(-50, -91){\line(0,-1){20}} \put(-50, -91){\line(-1,-1){20}}
\put(-50, -51){\circle*{2}}

\put(-28, -31){\line(0,1){110}} \put(-28, -31){\line(1,1){110}}
\put(-28, -71){\line(0,-1){40}} \put(-28, -71){\line(-1,-1){40}}
\put(-28, -31){\circle*{2}}

\put(-10, -31){\line(0,1){110}} \put(-10, -31){\line(1,1){110}}
\put(-10, -71){\line(0,-1){40}} \put(-10, -71){\line(-1,-1){40}}
\put(-10, -31){\circle*{2}}

\put(-8, -11){\line(0,1){90}} \put(-8, -11){\line(1,1){90}} \put(-8,
-51){\line(0,-1){60}} \put(-8, -51){\line(-1,-1){60}} \put(-8,
-11){\circle*{2}}

\put(12, -11){\line(0,1){90}} \put(12, -11){\line(1,1){90}} \put(12,
-51){\line(0,-1){60}} \put(12, -51){\line(-1,-1){60}} \put(12,
-11){\circle*{2}}

\put(30, -11){\line(0,1){90}} \put(30, -11){\line(1,1){70}} \put(30,
-51){\line(0,-1){60}} \put(30, -51){\line(-1,-1){60}} \put(30,
-11){\circle*{2}}

\multiput(-90,-51)(20,0){11}{\line(0,1){3}}
\multiput(-52,-90)(0,20){10}{\line(1,0){3}}
\put(-46,-57){$\scriptscriptstyle{0}$}
\put(-26,-57){$\scriptscriptstyle{1}$}
\put(-6,-57){$\scriptscriptstyle{2}$}
\put(14,-57){$\scriptscriptstyle{3}$}
\put(34,-57){$\scriptscriptstyle{4}$}
\put(-66,-57){$\scriptscriptstyle{-1}$}
\put(-86,-57){$\scriptscriptstyle{-2}$}
\put(-57,-47){$\scriptscriptstyle{0}$}
\put(-57,-27){$\scriptscriptstyle{1}$}
\put(-57,-7){$\scriptscriptstyle{2}$}
\put(-57,13){$\scriptscriptstyle{3}$}
%


\put(-58,95){${q}$} \put(125,-60){${p}$}
\end{picture}

\caption{$H^{*,*}(G_2(\R^{4,2}))$ with free generators shown.}
\label{fig:cohomg2r42}
\end{figure}
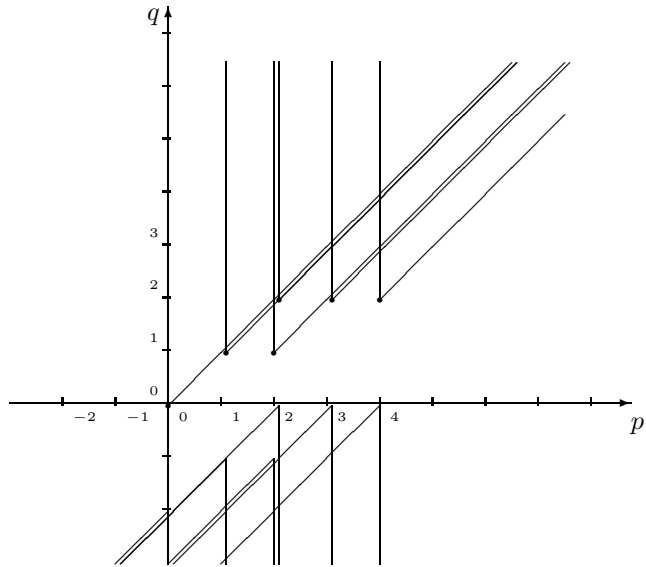

It should be noted that in the case of $G_2(\R^{4,1})$, with the
proper choice of flag symbols, the cell structure is such that the
differentials are all zero, and so the cohomology is free with
generators in the same degrees as the dimensions of the cells.  This
is \bf{not} \rm the case with $G_2(\R^{4,2})$.  Regardless of the
choice of flag symbol, there are some nonzero differentials which
cause some degree shifting of the cohomology generators.

Unfortunately, we cannot play this game indefinitely.  For the
Grassmann manifolds $G_n (\R^{p,q})$ with $n$ and $q$ small enough,
say $n\leq 2$ and $q\leq 2$, the above techniques can be used to
obtain the additive structure of $H^{*,*}(G_n (\R^{p,q}))$. However,
there are examples where the precise degrees of the cohomology
generators cannot be determined by comparing the cellular spectral
sequences for various flag symbols.  A serious inquiry into the
geometry of the attaching maps in these cell structures may reveal
more information.
\newpage
\bibliographystyle{alpha}
\bibliography{references}

\end{document}